\documentclass[12pt,a4]{article}
\usepackage{latexsym}
\usepackage{amsmath}
\usepackage{amssymb}
\usepackage{amsthm}
\usepackage{amscd}

\newtheorem{theorem}{Theorem}[section]
\newtheorem{lemma}[theorem]{Lemma}
\newtheorem{corollary}[theorem]{Corollary}
\newtheorem{proposition}[theorem]{Proposition}

\theoremstyle{definition}

\newtheorem{remark}[theorem]{Remark}
\newtheorem{definition}[theorem]{Definition}

\theoremstyle{remark}

\newcommand{\cS}{{\cal S}}

\newcommand{\D}{{\mathbb D}}
\newcommand{\T}{{\mathbb T}}

\renewcommand{\P}{{\mathbb P}}
\newcommand{\R}{{\mathbb R}}
\newcommand{\N}{{\mathbb N}}
\newcommand{\Z}{{\mathbb Z}}
\newcommand{\C}{{\mathbb C}}
\newcommand{\Q}{{\mathbb Q}}

\begin{document}

\title{Maximal operator\\
for pseudo-differential operators\\
with homogeneous symbols}

\author{Yoshihiro Sawano}

\maketitle

\vspace{1.5cm}

\noindent{\bf $2000$ Mathematics Subject Classification}\ \
Primary 42B10\,; Secondary 47B38.

\begin{abstract}
The aim of the present paper
is to obtain a Sj\"{o}lin-type maximal estimate
for pseudo-differential operators
with homogeneous symbols.
The crux of the proof is to obtain
a phase decomposition formula
which does not involve the time traslation.
The proof is somehow parallel to the paper
by Pramanik and Terwilleger
(P.~Malabika and E.~Terwilleger,
A weak $L^2$ estimate for a maximal dyadic sum operator
on $R^n$,
Illinois J.~Math,
{\bf 47} (2003), no. 3, 775--813).
In the present paper,
we mainly concentrate on our new phase decomposition formula
and the results
in the Cotlar type estimate,
which are different from the ones
by Pramanik and Terwilleger.
\end{abstract}

\section{Introduction}
\noindent

The class $S^0$ is a basic class of pseudo-differential operators,
which has been investigated by many authors.
For example,
it is quite fundamental that the pseudo-differential operators
with symbol $S^0$ are $L^2$-bounded (see \cite{St}).
However, in view of the fact that $L^2 \simeq \dot{F}^0_{22}$,
where $\dot{F}^0_{22}$ is the homogeneous Triebel-Lizorkin space,
there seems to be no need that we assume
\[
\sup_{x \in \R^n, \, |\xi| \le 1}
|\partial^\alpha_\xi \partial^\beta_x a(x,\xi)|<\infty 
\]
for all multiindices $\alpha,\beta$.
Indeed, Grafakos and Torres established
that it suffices to assume
\begin{equation}
\label{eq:c-alpha}
c_{\alpha,\beta}(a)
:=
\sup_{x \in \R^n, \, \xi \in \R^n}
|\xi|^{|\alpha|-|\beta|}|\partial^{\alpha}_\xi\partial^{\beta}_x a(x,\xi)|
<\infty
\end{equation}
for all multiindices $\alpha,\beta$.
Denote by $a(x,D)^\sharp$
the formal adjoint of $a(x,D)$.
It is natural that we assume that
\begin{equation}
\label{eq:moment}
a(x,D)^\sharp 1(x)=0,
\end{equation} 
since one needs to postulate
some moment condition on atoms
for $\dot{F}^0_{22}$
when we consider the atomic decomposition (see \cite{Fr2,Tr4}).

Here and below,
we assume that 
$a \in L^\infty(\R^n \times \R^n) \cap 
C^\infty(\R^n \times (\R^n \setminus \{0\}))$
is a function satisfying (\ref{eq:c-alpha}) and (\ref{eq:moment}).
In \cite{Grafakos-Torres}
Grafakos and Torres established
that
\[
f \in {\cal S}_0
\mapsto
\int_{\R^n}
a(x,\xi)\exp(2\pi i x \cdot \xi){\cal F}^{-1}f(\xi)\,d\xi,
\]
extends to an $L^2$-bounded operator,
where ${\cal S}_0$ denotes
the closed subspace of ${\cal S}$
which consists of the functions
with vanishing moment of any order.

The aim of the present paper is to obtain
a maximal estimate related to this operator.
To formulate our results,
we need some notations.
Given $a,\xi \in \R^n$ and $\lambda>0$,
we define
\begin{align*}
T_a f(x)&:=f(x-a)\\
M_\xi f(x)&:=\exp(2\pi i \xi \cdot x)f(x)\\
D_\lambda f(x)&:=\lambda^{-n/2}f(\lambda^{-1} x).
\end{align*}
Here and below,
we use $A \lesssim_{X,Y,\cdots} B$
to denote that there exists a constant $c>0$
depending only on the parameters $X,Y,\cdots$
such that $A \le c\,B$.
If the constant $c$ above depends only on
$c_{\alpha,\beta}(a)$ and the dimension $n$,
we just write $A \lesssim B$.
If the two-sided estimate
$A \lesssim_{X,Y,\cdots} B \lesssim_{X,Y,\cdots} A$
holds,
then we write
$A \simeq_{X,Y,\cdots} B$.

In the present paper
we establish the following.
\begin{theorem}
\label{thm:main}
The following estimate holds{\rm \,:}
\[
\left|
\left\{
x \in \R^n
\, : \,
\sup_{\xi \in \Q^n}
|M_{-\xi}a(x,D)M_{\xi}f(x)|>\lambda
\right\}
\right|
\lesssim
\frac{1}{\lambda^2}\int_{\R^n}|f(x)|^2\,dx.
\]
\end{theorem}

We remark that 
the case when $a(x,\xi)=m(\xi)$ with $m$ homogeneous of degree $0$
was covered by Pramanik and Terwilleger \cite{PrTe}.

In the present paper, we prove Theorem \ref{thm:main}.
In Section \ref{section:phase decomposition},
we obtain a formula
of the Fourier multiplier.
The formula will be a simplification
of \cite{PrTe}
and enables us to extend the results in \cite{PrTe}.
What is new about this formula is
that there is no need to take average
over the time space, 
as will be alluded to in Section \ref{section:phase decomposition}.
We investigate an estimate of Cotlar type
in Section \ref{section:Cotlar}.
In Section \ref{section:proofs},
we shall prove Theorem \ref{thm:main}.
Our proof parallels the one in \cite{PrTe}.
So we will invoke their results and notations.
Finally in Section \ref{section:extension},
we consider an extension to $L^p$ $(1<p<\infty)$
of Theorem \ref{thm:main}.

\section{Preliminaries}
\label{section:preliminaries}
\noindent

Here and below we use the following notations.

\subsection{Notations on cubes}
\noindent

We begin with some notations
for $\R^n$.
\begin{definition}
\begin{enumerate}
\item
We denote $\{0,1,2,\cdots\}$ by $\N_0$.
\item
We equip $\R^n$
with the lexicographic order $\ll$.
Namely,
we define 
\[
x=(x_1,x_2,\cdots,x_n) \ll y=(y_1,y_2,\cdots,y_n), \,
x \ne y
\]
if and only if
$x_1=y_1, x_2=y_2, \cdots, x_{j-1}=y_{j-1}, \, x_j < y_j$
for some $j=1,2,\cdots,n$.
\item
We define
$\vec{\mathbf 1}:=(1,1,\cdots,1)$.
\end{enumerate}
\end{definition}

In the present paper,
we use the following notation
for dyadic cubes.
\begin{definition}
\begin{enumerate}
\item
By a dyadic cube,
we mean the one of the form
\[
Q_{\nu m}
:=\prod_{j=1}^n \left[\frac{m_j}{2^\nu},\frac{m_j+1}{2^\nu}\right)
\]
for $m=(m_1,m_2,\cdots,m_n)$ and $\nu \in \Z$.
We also define
its center and the side-length
by $\displaystyle c(Q_{\nu m}):=
\left(\frac{2m_1+1}{2^{\nu+1}},\cdots,\frac{2m_n+1}{2^{\nu+1}}\right)$
and
$\ell(Q_{\nu m}):=2^{-\nu}$.
\item
Given a dyadic cube $Q$,
we bisect $Q$ into $2^n$ cubes of equal length
and label them $Q_{(1)},Q_{(2)},\cdots,Q_{(2^n)}$ so that
\[
c(Q_{(1)}) \ll c(Q_{(2)}) \ll \cdots \ll c(Q_{(2^n)}).
\]
\end{enumerate}
\end{definition}

Unlike dyadic cubes,
we assume that the cubes are closed.
\begin{definition}
By a cube,
we mean the subset in $\R^n$
of the form
\[
Q(x,r)
:=
\left\{ y=(y_1,y_2,\cdots,y_n) \in \R^n \, : \,
\max_{i=1,2,\cdots,n}|x_i-y_i|
\le r
\right\}
\]
for $x=(x_1,x_2,\cdots,x_n)$ and $r>0$.
The center and the sidelength of $Q=Q(x,r)$ are given by
\[
c(Q):=x, \, \ell(Q):=2r.
\]
Given $\kappa>0$ and a cube $Q=Q(x,r)$,
we define
$\kappa\,Q:=Q(x,\kappa\,r)$.
\end{definition}

\subsection{Notations on tiles and trees}
\noindent

\begin{definition}
\begin{enumerate}
\item
By a tile we mean the closs product
of the form $s=Q_{\nu m} \times Q_{-\nu m'}$
with $\nu \in \Z$ and $m,m' \in \Z^n$.
Given such a tile $s$,
we define $I_s:=Q_{\nu m}$ and $\omega_s:=Q_{-\nu m'}$.
The set of all tiles will be denoted by $\D$.
\item
Let $u,v \in \D$.
Then we define $u \le v$
if and only if
$I_u \subset I_v, \, \omega_u \supset \omega_v$.
\item
A tree is a pair $(\T,t)$,
where $\T \subset \D$ is a finite subset of $\D$
and $t \in \D$ is a tile such that
$t \ge s$ for all $s \in \T$.
We define $\omega_\T:=\omega_t$
and $I_\T:=I_t$.
\item
Let $1 \le i \le 2^n$.
A tree $(\T,t)$ is called an $i$-tree,
if $\omega_{t(i)} \subset \omega_{s(i)}$
for all $s \in \T$.
\end{enumerate}
\end{definition}

Occasionally $t$ is called a top of $\T$.
Note that the top of $\T$ is not unique in general.
In the present paper,
to avoid confusion,
we call a pair $(\T,t)$ a tree in order to specify the top.

\subsection{Notations of auxiliary functions}
\noindent

Here and below
we assume that $\Phi \in \cS$ is a function
satisfying
\[
\chi_{Q(9/100)} \le \Phi \le \chi_{Q(1/10)}.
\]

\begin{definition}
\begin{enumerate}
\item
$\varphi:={\cal F}^{-1}\Phi$.
\item
$\Psi:=\Phi-\Phi(2 \cdot)$.
\item
Given a cube $Q$,
we define
$\displaystyle
\Phi_Q(\xi):=\Phi\left(\frac{\xi-c(Q)}{\ell(Q)}\right).
$
\item
\cite{PrTe}
Given a tile $s \in \D$,
we define
$\displaystyle
\varphi_s(x)
:=
M_{c(\omega_{s(1)})}T_{c(I_s)}D_{\ell(I_s)}\varphi(x).
$
\end{enumerate}
\end{definition}

The following property is easily shown.
\begin{lemma}
\begin{enumerate}
\item
Let $Q$ be a cube.
Then we have
\begin{equation}
\label{eq:080823-1}
\chi_{\frac{27}{25}Q} \le \Phi_{6Q} \le \chi_{\frac{6}{5}Q}.
\end{equation}
\item
Let $s$ be a tile.
Then we have
\begin{equation}
\label{eq:080823-2}
{\cal F}\varphi_s
=T_{c(\omega_{s(1)})}M_{-c(I_s)}D_{\ell(\omega_s)}\Phi.
\end{equation}
In particular,
$\displaystyle
{\rm supp}({\cal F}\varphi_s)
\subset
\frac{1}{5}\omega_{s(1)}.
$
\end{enumerate}
\end{lemma}

(\ref{eq:080823-1}) says that $\Phi_{6Q}$
is almost the same as $\chi_Q$.
Meanwhile
(\ref{eq:080823-2}) implies that 
the frequency support of $\varphi_s$ is concentrated
near $c(\omega_{s(1)})$.

The following lemma is easy to show
by using the Planchrel theorem.
\begin{lemma}
\label{lem:080827-1}
Let $\xi \in \R^n$.
Then we have
\[
\left(\sum_{s \in \D\,:\,\omega_{s(2^n)} \ni \xi}
|\langle f,\varphi_s \rangle_{L^2}|^2\right)^\frac12
\lesssim
\| f \|_2.
\]
\end{lemma}

Next,
we consider the model operator.
\begin{definition}
The (model) dyadic operator is given by
\[
A_{\xi,\P} f(x)
:=
\sum_{s \in \P\,:\,, \omega_{s(2^n)} \ni \xi}
\langle f,\varphi_s \rangle_{L^2}\varphi_s, \,
\P \subset \D, \, \xi \in \R^n.
\]
\end{definition}

\begin{lemma}[\cite{PrTe}]
\label{lem:080825-1}
$A_{\xi,\P}$ is $L^2$-bounded
uniformly over $\P \subset \D$ and $\xi \in \R^n${\rm\,:}
\[
\| A_{\xi,\P} \, : \, B(L^2) \|
\lesssim
1.
\]
\end{lemma}

\begin{proof}
It is convenient to rely on the molecular decomposition
described in \cite{Tr4}.
An alternative way to the proof is 
that we depend on the almost-orthogonality and Lemma \ref{lem:080827-1}.
\end{proof}

\subsection{Integral kernel of $a(x,D)$}
\noindent

We define
\[
a_j(x,D)f(x)
:=
\int_{\R^n \times \R^n}a(x,\xi)\Psi(2^{-j}\xi)\exp(2\pi i \xi \cdot (x-y))
f(y)\,dy\,d\xi, \, j \in \Z^n,
\]
where we have defined $\Psi=\Phi-\Phi(2\cdot)$.
Then we have
\[
a_j(x,D)f(x)=\int k_j(x,x-z)f(z)\,dz.
\]
The integral kernel can be written as
\[
k_j(x,z)
:=
\int_{\R^n}
a(x,\xi)\Psi(2^{-j}\xi)\exp(2\pi i \xi \cdot z)\,d\xi.
\]
It is not so hard to show
the following estimate using integration by parts.
\begin{lemma}
Let $\alpha,\beta \in \N_0{}^n$.
Then we have
\[
|\partial^\alpha_x \partial^\beta_z k_j(x,z)|
\lesssim_{\alpha,\beta,L}
\min(2^{j(n+|\alpha|+|\beta|)},2^{j(n+|\alpha|+|\beta|)-2j L}|z|^{-2L}).
\]
\end{lemma}
A direct consequence of this lemma is that
\[
\sum_{j=-\infty}^\infty
|\partial^\alpha_x \partial^\beta_z k_j(x,z)|
\lesssim_{\alpha,\beta}
|z|^{-(n+|\alpha|+|\beta|)}.
\]
Let us set
$\displaystyle
k(x,z)
:=
\sum_{j=-\infty}^\infty k_j(x,z)
$
and write $a(x,D)$ as
\[
a(x,D)f(x)=\int k(x,x-z)f(z)\,dz, \, x \notin {\rm supp}(f)
\]
in terms of the integral kernel.
Recall that $a(x,D)$ is proved to be $L^2$-bounded
(see \cite{Grafakos-Torres}).
As a consequence we have
\begin{equation}
\label{eq:080823-3}
\int_{\R^n}
\sup_{\varepsilon>0}
\left|
a(x,D)[\chi_{\R^n \setminus Q(x,\varepsilon)}f](x)
\right|^2\,dx
\lesssim
\int_{\R^n}
|f(x)|^2\,dx.
\end{equation}
(\ref{eq:080823-3}) is known as
the maximal estimate
of the truncated singular integral operator
(see \cite{St}).

\section{Simplified phase decomposition formula
and some reductions of Theorem \ref{thm:main}}
\label{section:phase decomposition}
\noindent

In this section,
based on the notation in Section \ref{section:preliminaries},
we obtain a simplified phase decomposition formula.

\subsection{Simplified phase decomposition formula}

\begin{definition}
The model operator $A_{\eta,l}$ of the $l$-th generation
is defined by
\[
A_{\eta,l} f(x):=
\sum_{s \in \D \,:\, \omega_{s(2^n)} \ni \eta, \, |I_s|=2^{ln}}
\langle f,\varphi_s \rangle_{L^2} \varphi_s.
\]
\end{definition}

\begin{lemma}
\label{lem:2-1}
There exists a function $m \in C^\infty(\R^n \setminus \{0\})$
such that
\[
\lim_{N \to \infty}
\int_{Q_N}
M_{-\eta}A_{\eta,l}M_{\eta}f\,\frac{d\eta}{|Q_N|}
=
{\cal F}^{-1}[m(2^{l}\cdot) \cdot {\cal F}f]
\]
for any sequence of cubes $\{Q_N\}_{N \in \N}$
such that $2Q_N \subset Q_{N+1}$ for all $N \in \N$,
where the convergence takes place
in the strong topology of $L^2$.
\end{lemma}

\begin{proof}
The family of operators
\[
\left\{
\int_{Q_N}
M_{-\eta}A_{\eta,l}M_{\eta}\,\frac{d\eta}{|Q_N|}
\right\}_{N \in \N}
\]
being uniformly bounded in $B(L^2)$,
we can assume that $f \in {\cal S}_0$
to investigate the limit as $N \to \infty$.
Let us consider
\[
{\cal F}
\left(
\int_{Q_N}
M_{-\eta}A_{\eta,l}M_{\eta}\,\frac{d\eta}{|Q_N|}
\right){\cal F}^{-1}f
=
\int_{Q_N}
{\cal F}M_{-\eta}A_{\eta,l}M_{\eta}{\cal F}^{-1}f\,\frac{d\eta}{|Q_N|}.
\]
Let us denote by $Q_l=Q_l(\eta)$
the unique dyadic cube with $\ell(Q_l)=2^{-l}$
such that $\eta \in Q_{l(2^n)}$,
if there exists.
By using the Fourier expansion and (\ref{eq:080823-2}),
we have, assuming the existence of such $Q_l$,
\begin{align*}
\lefteqn{
{\cal F}M_{-\eta}A_{\eta,l}M_{\eta}{\cal F}^{-1}f
}\\
&=
\sum_{s \in \D\,:\, \omega_{s(2^n)} \ni \eta, \, |I_s|=2^{ln}}
\langle M_{\eta}{\cal F}^{-1}f,\varphi_s \rangle_{L^2} 
{\cal F}M_{-\eta}\varphi_s\\
&=
\sum_{s \in \D\,:\, \omega_{s(2^n)} \ni \eta, \, |I_s|=2^{ln}}
\langle f,{\cal F}M_{-\eta}\varphi_s \rangle_{L^2} 
{\cal F}M_{-\eta}\varphi_s\\
&=
\sum_{s \in \D\,:\, \omega_{s(2^n)} \ni \eta, \, |I_s|=2^{ln}}
\langle f,T_{c(\omega_{s(1)})-\eta}
M_{c(I_s)}
D_{\ell(\omega_s)}\Phi \rangle_{L^2} 
T_{c(\omega_{s(1)})-\eta}
M_{c(I_s)}
D_{\ell(\omega_s)}\Phi\\
&=
f
\cdot
\left|
\Phi\left(\frac{\cdot+\eta-c(Q_{l(1)})}{\ell(Q_l)}\right)
\right|^2.
\end{align*}
Inserting this equality,
we obtain
\[
{\cal F}
\left(
\int_{Q_N}
M_{-\eta}A_{\eta,l}M_{\eta}\,\frac{d\eta}{|Q_N|}
\right){\cal F}^{-1}f
=
m_l
\cdot
f,
\]
where
\[
m_l:=
2^{ln}
\int_{Q_{l+1,\vec{\mathbf 1}}}
\left|\Phi\left(2^l(\cdot+\eta-2^{-2}\vec{\mathbf 1})\right)\right|^2
\,d\eta
=
\int_{\frac{\vec{\mathbf 1}}{2}+Q\left(\frac14\right)}
\left|\Phi\left(2^l\cdot+\zeta\right)\right|^2
\,d\zeta.
\]
Hence,
we have the desired result
with
$\displaystyle
m:=
\int_{\frac{\vec{\mathbf 1}}{2}+Q\left(\frac14\right)}
\left|\Phi\left(\cdot+\zeta\right)\right|^2
\,d\zeta.
$
\end{proof}

\begin{corollary}
\label{cor:2-2}
Keep to the same notation as Lemma \ref{lem:2-1}.
Define
\begin{equation}
\label{eq:080829-1}
M(\xi):=\sum_{l=-\infty}^\infty m(2^l \xi).
\end{equation}
Then we have
\[
\lim_{L \to \infty}
\sum_{l=-L}^L
\left(
\lim_{N \to \infty}\int_{Q_N}
M_{-\eta}A_{\eta,l} M_{\eta}f\,\frac{d\eta}{|Q_N|}
\right)
=
{\cal F}^{-1}(M \cdot {\cal F}f),
\]
where the convergence takes place in the strong topology of $L^2$.
\end{corollary}

With this result,
we can obtain
a (simpler) decomposition of the phase space.
Recall that $SO(n)$ denotes
the set of all orthogonal matrices
with determinant $1$.
Since $SO(n)$ is compact,
it carries the normalized Haar measure $\mu$.
We define $\rho:SO(n) \to U(L^2)$
as the unitary representation of $SO(n)$,
namely, we define
\[
\rho(A)f:=f(A^{-1}\cdot), \, f \in L^2.
\]
\begin{corollary}
\label{cor:2-3}
Keep to the same notation as Lemma \ref{lem:2-1}.
Let $\alpha>0$ be a constant given by
\[
\alpha:=\int_{SO(n)}\int_0^1 M(2^\kappa A\,\xi)\,d\kappa\,d\mu
\]
for $\xi \in \R^n \setminus \{0\}$.
Then we have
\begin{align*}
\lefteqn{
\alpha\,id_{L^2}
}\\
&=
\int_{SO(n)}\int_0^1
\left(
\sum_{l=-\infty}^\infty
\lim_{N \to \infty}
\int_{Q_N}
\rho(A^{-1})D_{2^{-\kappa}}
M_{-\eta}A_{\eta,l} M_{\eta}D_{2^\kappa}\rho(A)\,\frac{d\eta}{|Q_N|}
\right)\,d\kappa\,d\mu,
\end{align*}
where all the convegences take place in the strong topology of $L^2$.
\end{corollary}

\begin{remark}
\begin{enumerate}
\item
In view of (\ref{eq:080829-1}),
$\alpha$ does not depend on $\xi$
appearing in the definition
of the formula defining $\alpha$.
\item
In \cite{PrTe}
Plamanik and Terwilleger considered
the average of
\[
\rho(A^{-1})D_{2^{-\kappa}}
T_{-y} M_{-\eta}A_{\eta,l} M_{\eta}T_y D_{2^\kappa}\rho(A).
\]
However, as our Corollary \ref{cor:2-3} shows,
there is no need to take average over the time space $\R^n_y$.
We shall take full advantage of this fact
in the course of the proof of Theorem \ref{thm:main}.
\end{enumerate}
\end{remark}

\subsection{Some reductions of Theorem \ref{thm:main}}
\noindent

Corollary \ref{cor:2-3} is the simplified phase decomposition formula,
which is beautiful of its own right.
However, in the present paper,
we discretize it.
More precisely,
we proceed as follows\,:
\begin{proposition}
\label{prop:2-4}
Let $\{A_n\}_{n \in \N}$ and $\{\kappa(n)\}_{n \in \N}$
be dense subsets of $SO(n)$ and $[0,1]$ respectively
such that $A_1=id_{\R^n}$ and $\kappa(1)=0$.
Then
\begin{equation}
\label{eq:080829-2}
m_K(\xi)
:=
\sum_{k_1,k_2=1}^K M(2^{\kappa(k_1)} A_{k_2}^{-1}\xi)
\end{equation}
satisifies the following conditions, provided $K$ is sufficiently large.
\begin{enumerate}
\item
$c_{\alpha,\beta}(m_K)<\infty$ for all $\alpha,\beta \in \N_0{}^n$.
\item
$\displaystyle \inf_{\xi \in \R^n \setminus \{0\}}m_K(\xi)>0$.
\end{enumerate}
\end{proposition}

\begin{proof}
This is clear from the definition of $M$.
\end{proof}

In view of this proposition,
we set $b(x,\xi):=a(x,\xi)/m_K(\xi)$.
Then we have
\begin{align*}
\lefteqn{
a(x,D)
}\\
&=
\sum_{k_1,k_2=1}^K
\sum_{l=-\infty}^\infty
\lim_{N \to \infty}
\int_{K_N}
b(x,D)\rho(A_{k_1}^{-1})D_{2^{-\kappa(k_2)}}
M_{-\eta}A_{\eta,l} M_{\eta}D_{2^{\kappa(k_2)}}\rho(A_{k_1})
\,\frac{d\eta}{|K_N|}.
\end{align*}
Since other summand can be dealt similarly,
let us consider the summand for $k_1=k_2=1$\,:
Below we shall deal with
\begin{align*}
\lefteqn{
\sum_{l=-\infty}^\infty
\lim_{N \to \infty}
\int_{K_N}
b(x,D)
M_{-\eta}A_{\eta,l} M_{\eta}\,\frac{d\eta}{|K_N|}
}\\
&=
\sum_{l=-\infty}^\infty
\lim_{N \to \infty}
\int_{K_N}
M_{-\eta}
b(x,D-\eta)A_{\eta,l} M_{\eta}\,\frac{d\eta}{|K_N|}.
\end{align*}
Recall that 
the main theorem concerns the conjugated modulation.
So, we are led to consider
\begin{align*}
\lefteqn{
\sum_{l=-\infty}^\infty
\lim_{N \to \infty}
\int_{K_N}
M_{-\xi-\eta}
b(x,D-\eta)A_{\eta,l} M_{\eta+\xi}\,\frac{d\eta}{|K_N|}
}\\
&=
\sum_{l=-\infty}^\infty
\lim_{N \to \infty}
\int_{K_N}
M_{-\eta}
b(x,D-\eta+\xi)A_{\eta-\xi,l}M_{\eta}\,\frac{d\eta}{|K_N|}.
\end{align*}
Here the equality holds
by virtue of Lemma \ref{lem:2-1}.

Define a norm by
\[
\| f \, : \, L^{2,\infty} \|^*
:=
\sup_{E}
|E|^{-\frac12}\int_E |f|.
\]
Here $E$ in $\sup$ runs 
over all the non-empty bounded measurable sets.
Then, the weak-$L^2$ quasi-norm 
is equivalent to this norm (see \cite{Grafakos}).
Furthermore, if $f$ is locally square integrable,
then we have
\begin{equation}
\label{eq:080815-1}
\| f \, : \, L^{2,\infty} \|^*
\simeq
\sup_{E}
|E|^{-\frac12}\left|\int_E f\right|.
\end{equation}
In view of Proposition \ref{prop:2-4}
the functions $a$ and $b$ enjoy the same property\,:
\[
c_{\alpha,\beta}(a) \simeq_{\alpha,\beta} c_{\alpha,\beta}(b)
\]
for all $\alpha,\beta \in \N_0{}^n$.
Hence, it is sufficient to show that
\[
\sup_{E}
|E|^{-\frac12}\int_E 
\sup_{\xi \in \Q^n}
\left|\sum_{l=-L}^L a(x,D-\xi)A_{\xi,l} f(x)\right|\,dx
\lesssim
\| f \|_2.
\]
Since there exists a measurable mapping
$N:\R^n \to \Q^n$
such that
\[
\sup_{\xi \in \Q^n}
\left|\sum_{l=-L}^L a(x,D-\xi)A_{\xi,l} f(x)\right|
\le 2
\left|\sum_{l=-L}^L a(x,D-\xi)A_{\xi,l} f(x)|_{\xi=N(x)}\right|,
\]
we have only to show
\[
\sup_{E}
|E|^{-\frac12}\left|\int_E 
\sum_{l=-L}^L a(x,D-\xi)A_{\xi,l} f(x)_{|\xi=N(x)}\,dx\right|
\lesssim
\| f \|_2.
\]
Taking into account
Lemma \ref{lem:080824-7} below,
we conclude that
\[
a(x,D-\xi)\sum_{l=-L}^L A_{\xi,l}f(x)
=
\lim_{M \to \infty}
\sum_{\substack{s \in \D\,:\, \omega_{s(2^n)} \ni \xi \\ 
2^{-L} \le \ell(I_s) \le 2^L, \,
|c(I_s)| \le M}}
\langle f,\varphi_s \rangle_{L^2}a(x,D-\xi)\varphi_s
\]
converges pointwise.
Hence, we have only to establish
that
\[
\sup_{\P \subset \D}
\left|
\sum_{s \in \P}
\langle f,\varphi_s \rangle_{L^2}
\int_{\R^n}
\chi_{N^{-1}[\omega_{s(2^n)}] \cap E}(x)
a(x,D-\xi)
\varphi_s(x)_{|\xi=N(x)}\,dx\right|
\lesssim
|E|^{\frac12}\| f \|_2,
\]
where $\P \subset \D$ runs over any finite set.
Finally by scaling we can assume
that $|E| \le 1$.
We refer to \cite[p780]{PrTe} for more details of this dilation technique.

With this in mind,
we are going to prove the following
in Section \ref{section:proofs}.
\begin{theorem}[Basic estimate]
\label{thm:main2}
Let $N:\R^n \to \R^n$ be a measurable mapping
and $E$ a bounded measurable subset whose volume is less than $1$.
Then we have
\[
\sum_{s \in \D}
\left|
\langle f,\varphi_s \rangle_{L^2}
\int_{\R^n}
\chi_{N^{-1}[\omega_{s(2^n)}] \cap E}(x)
a(x,D-\xi)
\varphi_s(x)_{|\xi=N(x)}\,dx\right|
\lesssim
\| f \|_2.
\]
\end{theorem}

Below in the present paper
we fix a measurable mapping $N:\R^n \to \R^n$
and a bounded measurable set $E$ with volume less than $1$.
To simplify notations,
we define
$E_{s(2^n)}:=N^{-1}[\omega_{s(2^n)}] \cap E$
and
\[
\psi_s^\xi(x):=a(x,D-\xi)\varphi_s(x), \,
\psi_s^{N(\cdot)}(x):=\psi_s^\xi(x)|_{\xi=N(x)}.
\]
As for $\psi_s^\xi$,
we have the following pointwise estimate.
\begin{lemma}
\label{lem:080824-7}
$\displaystyle
|\psi_s^\xi(x)|
\lesssim_L
|I_s|^{-\frac{1}{2}}
\left(1+\frac{|x-c(I_s)|}{\ell(I_s)}\right)^{-L}
$
for all $L \in \N$.
\end{lemma}

Following the notation in \cite{Lacey-Terwilleger},
we define
\[
{\rm Sum}(\P):=
\sum_{s \in \P}
|\langle f,\varphi_s \rangle_{L^2}|
\cdot
|\langle \psi_s^{N(\cdot)},\chi_{E_{s(2^n)}}\rangle_{L^2}|
\]
for $\P \subset \D$.
Finally
we shall establish
\begin{equation}
\label{eq:080816-1}
{\rm Sum}(\P)
\lesssim
\| f \|_2
\end{equation}
for any finite subset $\P$ 
instead of proving Theorem \ref{thm:main2}
directly.

\section{Cotlar type estimate}
\label{section:Cotlar}

In this section we obtain a Cotlar type estimate.
We let
\begin{align*}
a_{\eta,\tau,\ell}(x,\xi)
&:=
a(x,\xi-\eta)
\Phi\left(\frac{\xi-\tau}{6\ell}\right), \,
a_{s,\eta}(x,\xi):=a_{\eta,c(\omega_s),\ell(\omega_s)}(x,\xi)
\end{align*}
for $\ell>0$, $\eta,\tau \in \R^n$ and $s \in \D$.
To formulate our result,
we use the maximal operator $M_{\ge b}$ by
\[
M_{\ge b}f(x)
:=
\sup_{r \ge b}
\frac{1}{r^n}\int_{Q(x,r)}|f(y)|\,dy
=
\sup_{r \ge b}
\frac{1}{r^n}\int_{Q(r)}|f(x+y)|\,dy
\]
for $b>0$.
We prove the following estimate.
\begin{proposition}
\label{prop:Cotlar}
Let $u,v \in \D$ with $u \le v$.
Suppose that $y \in \R^n$ and $\eta_0,\eta_1 \in \omega_v$.
Then we have
\begin{align*}
\lefteqn{
|a_{v,\eta_0}(x,D)f(y)-a_{u,\eta_0}(x,D)f(y)|
}\\
&\lesssim
\inf_{z \in Q(y,\ell(I_u))}
\left(
M_{\ge \ell(I_u)}f(z)
+
\sup_{\varepsilon>0}
\left|a(x,D-\eta_1)[\chi_{\R^n \setminus Q(z,\varepsilon)}f](z)\right|
\right).
\end{align*}
\end{proposition}

\subsection{Maximal operator $M_{\ge b}$}
\noindent

In the present section
we frequently use the following estimates.
\begin{lemma}
\begin{enumerate}
\item
Let $a>0$ and $L>n$.
Then we have
\begin{equation}
\label{eq:080824-1}
\int_{\R^n \setminus Q(x,a)}
\frac{a^{L-n}|f(y)|}{|x-y|^L}\,dy
\lesssim_L
M_{\ge a}f(x).
\end{equation}
\item
Let $b>a>0$.
Then we have
\begin{equation}
\label{eq:080824-2}
\int_{Q(x,b) \setminus Q(x,a)}
\frac{|f(y)|}{b\,|x-y|^{n-1}}\,dy
\lesssim
M_{\ge a}f(x).
\end{equation}
\end{enumerate}
\end{lemma}

\begin{proof}
For the proof of (\ref{eq:080824-2}),
we may assume that $a=2^{-l}b$ for some $l \in \N$
by replacing $a$ with a number slightly less than $a$.
Both cases can be proved easily by decomposing
\[
\int_{\R^n \setminus Q(x,a)}
=
\sum_{j=1}^\infty
\int_{Q(x,2^j a) \setminus Q(x,2^{j-1}a)},
\int_{Q(x,b) \setminus Q(x,a)}
=
\sum_{j=1}^l
\int_{Q(x,2^{1-j}b) \setminus Q(x,2^{-j}b)}.
\]
Using this decomposition,
we can prove (\ref{eq:080824-1}) and (\ref{eq:080824-2})
easily.
We omit the further details.
\end{proof}

\begin{lemma}
\label{lem:Cotlar1}
Let $a,b>0$, $s \in \D$ and $y,y^*,\eta,\tau \in \R^n$.
Then we have
\begin{align}
\label{eq:080824-3}
|a_{\eta,\tau,a\,\ell(I_s)}(x,D)[\chi_{Q(y,\ell(I_s))}f](y^*)|
\lesssim_{a,b}
M_{\ge \ell(I_s)}f(y),
\end{align}
and
\begin{align}
\label{eq:080824-4}
|(a(x,D-\eta)-a_{\eta,\tau,\ell(I_s)}(x,D))
[\chi_{\R^n \setminus Q(y,\ell(I_s))}f](y)|
\lesssim
M_{\ge \ell(I_s)}f(y),
\end{align}
whenever $|y-y^*| \lesssim_b \ell(I_s)$.
\end{lemma}

\begin{proof}
By the triangle inequality
we have
\[
\mbox{L.H.S. of }(\ref{eq:080824-3})
\le
\int_{Q(y,\ell(I_s))}
\left(
\int_{\R^n}\left|\Phi\left(\frac{\xi-c(\omega_u)}{6a\,\ell(\omega_u)}\right)
\right|\,d\xi\right)|f(z)|\,dz,
\]
from which we easily obtain (\ref{eq:080824-3}).

As for (\ref{eq:080824-4}),
we decompose
\[
a(x,D-\eta)-a_{\eta,\tau,\ell(I_s)}(x,D)
=
\sum_{j=1}^\infty
a_{\eta,\tau,2^j\ell(I_s)}(x,D)
-
a_{\eta,\tau,2^{j-1}\ell(I_s)}(x,D).
\]
Observe that the integral kernel $k_j(x,z)$
of $
a_{\eta,\tau,2^j\ell(I_s)}(x,D)
-
a_{\eta,\tau,2^{j-1}\ell(I_s)}(x,D)$
has the following bound
\[
|k_j(x,z)| \lesssim_L
(2^j\ell(I_s))^{n-2L}|x-z|^{-2L}
\]
for each $L \in \N$.
This inequality is summable, if $L=n$,
and we obtain
\begin{align*}
\mbox{L.H.S. of }(\ref{eq:080824-4})
&\lesssim
\int_{\R^n \setminus Q(y,\ell(I_s))}
\frac{|f(z)|\,dz}{\ell(I_s)^{n-2L}|z-y|^{2L}}
\lesssim
M_{\ge \ell(I_s)}f(y).
\end{align*}
Thus, the proof of (\ref{eq:080824-4}) is now complete.
\end{proof}

The following estimate can be obtained
by the same idea as (\ref{eq:080824-4}).
\begin{lemma}
\label{lem:Cotlar3}
Suppose that $u \le v$ and $\eta \in \omega_v$.
Then we have
\[
|(a_{\eta,\eta,\ell(\omega_u)}(x,D)
-a_{\eta,\eta,\ell(\omega_v)}(x,D))[\chi_{\R^n \setminus Q(y,\ell(I_v))}f](y)|
\lesssim
M_{\ge \ell(I_u)}f(y).
\]
\end{lemma}

\begin{lemma}
\label{lem:Cotlar7}
Let $s \in \D$ and $\eta \in \omega_s$.
Then we have
\[
|a_{s,\eta}(x,D)f(y)-a_{\eta,\eta,\ell(I_s)}(x,D)f(y)|
\lesssim
M_{\ge \ell(I_s)}f(y)
\]
for all $y \in \R^n$.
\end{lemma}

\begin{proof}
The proof is straightforward
by using integration by parts.
\end{proof}

\subsection{Proof of Proposition \ref{prop:Cotlar}}
\noindent

Fix a point $z \in Q(y,\ell(I_u))$.
In view of Lemmas \ref{lem:Cotlar1}, \ref{lem:Cotlar3} and \ref{lem:Cotlar7},
it is sufficient to prove Proposition \ref{prop:Cotlar}
assuming that $f$ is supported
outside $Q(z,2\ell(I_u))$.
Note that
\[
M_{\ge \ell(I_u)}f(y) \simeq_\kappa M_{\ge \ell(I_u)}f(y^*)
\]
whenever $|y-y^*| \le \kappa\,b$.
Let us establish
\begin{align*}
|a_{\eta_0,\eta_0,\ell(I_v)}(x,D)f(y)
-a_{\eta_0,\eta_0,\ell(I_u)}(x,D)f(y)|
\lesssim
M_{\ge \ell(I_u)}f(y)
+
\left|a(x,D-\eta_1)f(z)\right|,
\end{align*}
which immediately yields Proposition \ref{prop:Cotlar}.
For the time being,
we concentrate on reducing the matter
to the case when $\eta_0=\eta_1$.

\begin{lemma}
\label{lem:Cotlar8}
Let $u \le v \in \D$ and $\eta_0,\eta_1 \in \omega_v$.
Set
\begin{align*}
\lefteqn{
A_{\eta_0,\eta_1,u,v}(x,D)
}\\
&:=a_{\eta_0,\eta_0,\ell(I_u)}(x,D)
-a_{\eta_1,\eta_1,\ell(I_u)}(x,D)
-a_{\eta_0,\eta_0,\ell(I_v)}(x,D)
+a_{\eta_1,\eta_1,\ell(I_v)}(x,D).
\end{align*}
Then we have
\[
|A_{\eta_0,\eta_1,u,v}(x,D)[\chi_{Q(y,\ell(I_v))}f](y)|
\lesssim
M_{\ge \ell(I_u)}f(y).
\]
\end{lemma}

\begin{proof}
Note that $A_{\eta_0,\eta_1,u,v}(x,D)$
can be written as
\begin{align*}
\lefteqn{
A_{\eta_0,\eta_1,u,v}(x,D)[\chi_{Q(y,\ell(I_v))}f](y)
}\\
&=
\sum_{j=1}^{\log_2 \frac{\ell(I_v)}{\ell(I_u)}}
\int_{Q(y,\ell(I_v)) \setminus Q(y,\ell(I_u))}
\left(
\int_{\R^n}
\alpha_j(y,y^*,\xi;\eta_0,\eta_1)\,d\xi
\right)f(y^*)\,dy^*,
\end{align*}
where
\begin{align*}
\alpha_j(y,y^*,\xi;\eta_0,\eta_1)
&:=-
a(y,\xi)
\Psi\left(\frac{\xi}{3 \cdot 2^{j+1}\ell(\omega_u)}\right)\\
&\quad \times
(\exp(2\pi i(\xi+\eta_0) \cdot (y-y^*))
-\exp(2\pi i(\xi+\eta_1) \cdot (y-y^*))).
\end{align*}
An integration by parts yields
\[
|\alpha_j(y,y^*,\xi;\eta_0,\eta_1)|
\lesssim_L
|y-y^*|^{1-2L}\ell(\omega_v)(2^j \ell(\omega_u))^{n-2L}
\]
for all $L \in \N$.
If $L>n/2$,
then this inequality is summable over $j \in \N$
and we obtain
\[
\sum_{j=1}^\infty
|\alpha_j(y,y^*,\xi;\eta_0,\eta_1)|
\lesssim_L
\ell(I_u)^{n+1}|y-y^*|^{-2n-1}.
\]

Inserting this estimate and invoking (\ref{eq:080824-1}),
we obtain
\[
|A_{u,v,\eta_0,\eta_1}(x,D)f(y)|
\lesssim
\int_{\R^n \setminus Q(y,\ell(I_u))}
\frac{\ell(I_u)|f(y^*)|}{|y-y^*|^{n+1}}\,dy^*
\lesssim
M_{\ge \ell(I_u)}f(y).
\]
Thus, the proof is therefore complete.
\end{proof}

\begin{corollary}
\label{cor:Cotlar8}
Suppose that $u \le v$ and $\eta_0,\eta_1 \in \omega_v$.
Then we have
\[
|(a_{u,\eta_0}(x,D)-a_{v,\eta_0}(x,D)
-a_{u,\eta_1}(x,D)+a_{v,\eta_1}(x,D))f(y)|
\lesssim
M_{\ge \ell(I_u)}f(y).
\]
\end{corollary}

\begin{proof}
Combine Lemmas \ref{lem:Cotlar1}, \ref{lem:Cotlar3} and \ref{lem:Cotlar8}.
\end{proof}

In view of Corollary \ref{cor:Cotlar8},
we can assume that $\eta_0=\eta_1=\eta \in \omega_v$,
which we shall do.

\begin{lemma}
\label{lem:Cotlar4}
Let $s \in \D$.
Then we have
\[
|a_{s,\eta}(x,D)f(y)|
\lesssim
M_{\ge \ell(I_s)}f(y)
+
|a(x,D-\eta)f(y)|
\]
for all $y \in \R^n$.
\end{lemma}

\begin{proof}
This is an immediate consequence
of (\ref{eq:080824-3}) and (\ref{eq:080824-4}).
\end{proof}

Proposition \ref{prop:Cotlar} will be proved completely
once we establish the following.

\begin{lemma}
\label{lem:Cotlar5}
Let $s \in \D$.
Then we have
\begin{align*}
|a(x,D-\eta)f(y)|
\lesssim
M_{\ge \ell(I_s)}f(y)
+
|a(x,D-\eta)f(z)|.
\end{align*}
for all $z \in Q(y,\ell(I_s))$.
\end{lemma}

\begin{proof}
We shall control
\[
|
a_{\eta,\eta,\ell(\omega_s)}(x,D)f(y)
-
a_{\eta,\eta,\ell(\omega_s)}(x,D)f(z)
|,
\]
which is sufficient by virtue of (\ref{eq:080824-4}).
Note that
\begin{align*}
a_{\eta,\eta,\ell(\omega_s)}(x,D)f(y)
-
a_{\eta,\eta,\ell(\omega_s)}(x,D)f(z)=
\int k(z^*)f(z^*)\,dz^*,
\end{align*}
where
\begin{align*}
k(z^*)
&:=
\int_{\R^n}
a(y,\xi-\eta)
\Phi\left(\frac{\xi-\eta}{6\ell(\omega_s)}\right)
\exp(2\pi i(y-z^*) \cdot \xi)\,d\xi\\
&\quad -
\int_{\R^n}
a(z,\xi-\eta)
\Phi\left(\frac{\xi-\eta}{6\ell(\omega_s)}\right)
\exp(2\pi i(z-z^*) \cdot \xi)\,d\xi.
\end{align*}
Let us define
\begin{align*}
k_j(z^*)
&:=-
\int_{\R^n}
a(y,\xi-\eta)
\Psi\left(\frac{\xi-\eta}{3 \cdot 2^{3-j}\ell(\omega_s)}\right)
\exp(2\pi i(y-z^*) \cdot \xi)\,d\xi\\
&\quad +
\int_{\R^n}
a(z,\xi-\eta)
\Psi\left(\frac{\xi-\eta}{3 \cdot 2^{3-j}\ell(\omega_s)}\right)
\exp(2\pi i(z-z^*) \cdot \xi)\,d\xi.
\end{align*}
Then we have $\displaystyle k=\sum_{j=1}^\infty k_j$.

A simple calculation yields
\[
|k_j(z^*)|
\lesssim_L
\ell(I_s)(2^{-j}\ell(\omega_s))^{n+1-2L}
|y-z^*|^{-2L}
\]
for all $L \in \N$.
Interpolating this inequality with $L=0,n+1$,
we obtain
\[
|k_j(z^*)|
\lesssim_\theta
\ell(I_s)(2^{-j}\ell(\omega_s))^{1-\theta}
|y-z^*|^{-n-\theta}
\]
for $0<\theta<1$
and hence
\[
\sum_{j=1}^\infty|k_j(z^*)|
\lesssim_\theta
\ell(\omega_s)^{-\theta}|y-z^*|^{-n-\theta}.
\]
As a result,
we obtain
\[
|a_{\eta,\eta,\ell(\omega_s)}(x,D)f(y)
-a_{\eta,\eta,\ell(\omega_s)}(x,D)f(z)|
\lesssim
M_{\ge \ell(I_s)} f(y).
\]
This is the desired result.
\end{proof}

\section{Proofs of Theorems \ref{thm:main} and \ref{thm:main2}}
\label{section:proofs}
\noindent

In this section
we shall prove Theorems \ref{thm:main} and \ref{thm:main2}
 which are reduced
to establishing (\ref{eq:080816-1}).

\subsection{Review of ${\rm Size}$ and ${\rm Count}$}
\noindent

\begin{definition}[\cite{Lacey,Lacey-Thiele,PrTe}]
\begin{enumerate}
\item
The density of a tile $s \in \D$
is defined by
\[
{\rm dense}(s)
:=
\int_{E \cap N^{-1}[\omega_s]}
\left(1+\frac{|x-c(I_s)|}{\ell(I_s)}\right)^{-20n}\,\frac{dx}{|I_s|}
\le
\left(\frac{2}{19}\right)^n.
\]
\item
Define
$\displaystyle
{\rm size}(\T_0)
:=
\left(\sum_{s \in \T_0}
\frac{|\langle f,\varphi_s \rangle_{L^2}|^2}{|I_t|}
\right)^\frac{1}{2}
$
for an $i$-tree $(\T_0,t)$ with $2 \le i \le 2^n$.
\end{enumerate}
\end{definition} 

\begin{definition}[\cite{Lacey,Lacey-Thiele,PrTe}]
Let $\P$ be a subset of $\D$.
Then define
\begin{align*}
{\rm Dense}(\P)
&:=
\sup_{s \in \P}{\rm dense}(s)\\
{\rm Size}(\P)
&:=
\sup\left\{
{\rm size}(\T_0)
\, : \, \T_0 \subset \P, \,
(\T_0,t) \mbox{ is an $i$-tree with $2 \le i \le 2^n$}
\right\}\\
{\rm Count}(\P)
&:=
\inf
\left\{
\sum_{j=1}^{J_0}|I_{t_j}|
\, : \, 
\mbox{each } (\T_j,t_j) \mbox{ is a tree and }
\P=\bigcup_{j=1}^{J_0}\T_j \mbox{ as a set }
\right\}.
\end{align*}
\end{definition}

We now invoke the following crucial lemmas.

\begin{lemma}{\rm \cite[Density lemma, Lemma 1]{PrTe}}
\label{lem:density}
There exists a constant $\alpha$
with the following property{\rm\,:}
Any finite subset $\T$
admits a partition such that
\[
\T=\T_{light} \coprod \T_{heavy}, \,
{\rm Dense}(\T_{light})\le\frac14{\rm Dense}(\T), \,
{\rm Count}(\T_{heavy}) \le \frac{\alpha}{{\rm Dense}(\T)}.
\]
\end{lemma}

\begin{lemma}{\rm \cite[Size lemma, Lemma 2]{PrTe}}
\label{lem:size}
There exists a constant $\beta$
with the following property{\rm\,:}
Any finite subset $\T$ admits a partition
such that
\[
\T=\T_{small} \coprod \T_{large}, \,
{\rm Size}(\T_{small})\le\frac12{\rm Size}(\T), \,
{\rm Count}(\T_{large}) \le \frac{\beta\| f \|_2{}^2}{{\rm Size}(\T)^2}.
\]
\end{lemma}

If we combine the density and the size lemma,
we obtain the following.
\begin{corollary}[{\rm \cite{Grafakos,PrTe}}]
\label{cor:density-size}
Any finite subset $\P \subset \D$
admits the following decomposition{\rm\,:}
\begin{enumerate}
\item
$\displaystyle \P=\coprod_{j=-\infty}^\infty \P_j$
\item
Set
$\displaystyle
{\mathbb U}_j:=\P \setminus \coprod_{k=j}^\infty \P_k.
$
Then we have
\begin{align*}
{\rm Dense}
\left({\mathbb U}_j\right) 
&\le 4^j\\
{\rm Size}
\left({\mathbb U}_j\right) 
&\le 2^j\| f \|_2.
\end{align*}
\item
${\rm Count}\left(\P_j\right) 
\le (\alpha+\beta)4^{-j}$.
\end{enumerate}
Here the constants $\alpha$ and $\beta$ are from Lemmas
\ref{lem:density} and \ref{lem:size}
respectively.
\end{corollary}

Although the proof is essentially contained in \cite{Grafakos,PrTe},
we outline the proof
for the sake of convenience for the readers.
\begin{proof}
Assume that
$j_0$ is large enough so that
\[
{\rm Dense}(\P) \le 4^{j_0}, \,
{\rm Size}(\P) \le 2^{j_0}\| f \|_2.
\]
We define
$\P_j:=\emptyset$
for $j \ge j_0$.
Assume that
$\P_k, \, k \ge j$ is defined so 
that 
\[
{\rm Dense}
\left({\mathbb U}_j\right)
\le 4^j, \,
{\rm Size}
\left({\mathbb U}_j\right)
\le 2^j\| f \|_2.
\]
We define $\P_{j-1}$
as follows\,:
Using Lemmas \ref{lem:density} and \ref{lem:size},
we define
\begin{align*}
\lefteqn{
\P_{j-1}
}\\
&:=
\left\{
\begin{array}{ll}
\emptyset 
& \mbox{, if }
{\rm Dense}({\mathbb U}_j) \le 4^{j-1}
\mbox{ and }
{\rm Size}({\mathbb U}_j)  \le 2^{j-1}\| f \|_2\\
({\mathbb U}_j)_{large}
& \mbox{, if }
{\rm Dense}({\mathbb U}_j) \le 4^{j-1}
\mbox{ and }
{\rm Size}({\mathbb U}_j) > 2^{j-1}\| f \|_2\\
({\mathbb U}_j)_{heavy} 
& \mbox{, if } 
{\rm Dense}({\mathbb U}_j) > 4^{j-1}
\mbox{ and }
{\rm Size}({\mathbb U}_j) \le 2^{j-1}\| f \|_2\\
({\mathbb U}_j)_{heavy} \cup ({\mathbb U}_j)_{large}
& \mbox{, if } 
{\rm Dense}({\mathbb U}_j) > 4^{j-1}
\mbox{ and }
{\rm Size}({\mathbb U}_j) > 2^{j-1}\| f \|_2.
\end{array}
\right.
\end{align*}
By Lemmas \ref{lem:density} and \ref{lem:size},
we see that 
${\rm Count}\left(\P_{j-1}\right) 
 \le 4(\alpha+\beta) \cdot 4^{-j}$
in any case.
\end{proof}

This is a quick review of the culmination
which will used for the proof 
of Theorems \ref{thm:main} and \ref{thm:main2}
in the present paper.

To prove (\ref{eq:080816-1}),
it suffices to establish
that
\begin{theorem}
\label{thm:main3}
There exists $\gamma$ with the following property{\rm\,:}
Let $(\T,t)$ be a tree.
Then we have
\begin{equation}
\label{eq:080816-2}
{\rm Sum}(\T)
\le \gamma{\rm Dense}(\T){\rm Size}(\T)|I_t|.
\end{equation}
In particular
\begin{equation}
\label{eq:080816-3}
{\rm Sum}(\P)
\le \gamma{\rm Dense}(\P){\rm Size}(\P){\rm Count}(\P).
\end{equation}
\end{theorem}

We remark that (\ref{eq:080816-3})
is an immediate consequence of (\ref{eq:080816-2}).
Indeed, to obtain (\ref{eq:080816-3})
we have only to decompose $\P$ into a sequence
of trees and add (\ref{eq:080816-2}) over such trees.
Furthermore,
once we obtain (\ref{eq:080816-3}),
we have
\begin{align*}
{\rm Sum}(\P_j)
\le
4\gamma(\alpha+\beta)\| f \|_2
\min(2^{-j},2^{j})
\end{align*}
under the notation in Corollary \ref{cor:density-size}.
This inequality is summable over $j \in \Z$
to yield Theorem \ref{thm:main2} and hence Theorem \ref{thm:main}.

By linearization,
(\ref{eq:080816-2}) amounts to establishing
\begin{equation}
\label{eq:080816-4}
\left|
\sum_{s \in \T}
\alpha_s
\langle f,\varphi_s \rangle_{L^2}
\cdot
\langle \psi_s^{N(\cdot)},\chi_{E_{s(2^n)}}\rangle_{L^2}
\right|
\lesssim
{\rm Dense}(\T){\rm Size}(\T)|I_t|
\end{equation}
for all sequences $\{\alpha_s\}_{s \in \T}
\subset \Delta(1):=\{z \in \C \, : \, |z|<1 \}$.

\subsection{Partition ${\cal J}(\T)$ of $\R^n$
and further reduction}
\noindent

To proceed, we consider a partition of $\R^n$ 
associated with a tree $\T$.

\begin{lemma}[\cite{Lacey-Thiele,PrTe}]
\label{lem:3-2}
Suppose that $\T$ is a tree.
Define
\[
{\cal J}_0(\T)
:=
\left\{
Q \in {\cal D} \, : \,
I_s \mbox{ is not contained in }3Q
\mbox{ for all } s \in \T
\right\}
\]
and ${\cal J}(\T)$ as the subfamily
which is made up of all cubes maximal with respect to inclusion.
Then ${\cal J}(\T)$ is a partition of $\R^n$.
\end{lemma}

It is not so hard to prove Lemma \ref{lem:3-2}
by using the maximality of ${\cal J}(\T)$.
Along with this partition,
(\ref{eq:080816-4}) can be decomposed into
\begin{align}
\label{eq:080816-5}
\sum_{J \in {\cal J}(\T)}
\left|
\sum_{s \in \T, \, |I_s|\le 2^n|J|}
\alpha_s
\langle f,\varphi_s \rangle_{L^2}
\cdot
\int_{J \cap E_{s(2^n)}}\psi_s^{N(\cdot)}(x)\,dx
\right|
&\lesssim
{\rm Dense}(\T){\rm Size}(\T)|I_t|\\
\label{eq:080816-6}
\sum_{J \in {\cal J}(\T)}
\left|
\sum_{s \in \T, \, |I_s|>2^n|J|}
\alpha_s
\langle f,\varphi_s \rangle_{L^2}
\cdot
\int_{J \cap E_{s(2^n)}}\psi_s^{N(\cdot)}(x)\,dx
\right|
&\lesssim
{\rm Dense}(\T){\rm Size}(\T)|I_t|.
\end{align}
Keeping Lemma \ref{lem:080824-7} in mind,
we can prove  (\ref{eq:080816-5}) completely
analogously to the corresponding part in \cite{PrTe}.
So, we omit the details.
For the proof of (\ref{eq:080816-6})
we need to utilize our simplified phase decomposition
formula.
Now we invoke the following result in \cite{PrTe}.
\begin{lemma}{\rm \cite[p795]{PrTe}}
$\displaystyle
\left|J \cap \bigcup_{s \in \T, \, |I_s|>2^n|J|}E_{s(2^n)}\right|
\lesssim
{\rm Dense}(\T)|J|.
$
\end{lemma}

\subsection{Conclusion of the proof of Theorem \ref{thm:main2}}
\noindent

To establish (\ref{eq:080816-6}),
we obtain a pointwise estimate of
\[
\sum_{s \in \T, \, |I_s|>2^n|J|}
\chi_{J \cap E_{s(2^n)}}(x)
\alpha_s
\langle f,\varphi_s \rangle_{L^2}
\psi_s^{N(\cdot)}(x).
\]
Below let us
set
\begin{align*}
F_1(x)
&:=
\sum_{s \in \T}
\alpha_s
\langle f,\varphi_s \rangle_{L^2}
\varphi_s(x)\\
F_{2,J}(x)
&:=
\sum_{s \in \T, \, |I_s|>2^n|J|}
\chi_{J \cap E_{s(2^n)}}(x)
\alpha_s
\langle f,\varphi_s \rangle_{L^2}
\psi_s^{N(\cdot)}(x).
\end{align*}
The following lemma is easy to show
with the help of Lemma \ref{lem:080825-1}.
\begin{lemma}[\cite{PrTe}]
\label{lem:080824-6}
$\displaystyle
\int_{\R^n}|F_1(x)|^2\,dx
\lesssim
|I_t|{\rm Size}(\T)^2.
$
\end{lemma}

To obtain the pointwise estimate,
we fix a point $x \in J$
such that $|I_s|>2^n|J|$
and $x \in E_{s(2^n)}$ for some $s \in \T$.

We define
\begin{align*}
\omega_+&=\omega_+(x;J)
:=
\bigcup\{ \omega_s \, : \, 
s \in \T, \, x \in E_{s(2^n)}, \, |I_s|>2^n|J| \}\\
\omega_-&=\omega_-(x;J)
:=
\bigcap\{ \omega_{s(2^n)} \, : \, 
s \in \T, \, x \in E_{s(2^n)}, \, |I_s|>2^n|J| \}.
\end{align*}
A geometric observation shows the following.
\begin{lemma}{\rm \cite{PrTe}}
Let $s \in \T$.
\begin{enumerate}
\item
If $\omega_+$ is a proper subset of $\omega_s$,
then we have
$\displaystyle \frac{6}{5}\omega_+ \cap \frac{1}{5}\omega_s=\emptyset$.
\item
$\omega_- \subsetneq \omega_s \subset \omega_+$
if and only if 
$|I_s|>2^n|J|$.
If this is the case,
then we have
$\displaystyle \frac{1}{5}\omega_s \subset \frac{27}{25}\omega_+
\setminus \frac{6}{5}\omega_-$.
\item
If $\omega_-$ contains $\omega_s$,
then we have
$\displaystyle \frac{1}{5}\omega_s \subset \frac{6}{5}\omega_-$.
\end{enumerate}
\end{lemma}

In view of this observation,
we have
\begin{align*}
F_{2,J}(x)
=
\left(
a_{\xi,c(\omega_+),\ell(\omega_+)}(x,D)
-
a_{\xi,c(\omega_-),\ell(\omega_-)}(x,D)
\right)F_1(x)|_{\xi=N(x)}.
\end{align*}

Let $\omega_+=\omega_u$ and $\omega_-=\omega_{v(2^n)}$
with $u,v \in \T$.
We apply Proposisition \ref{prop:Cotlar}
with $\eta_0=N(x)$ and $\eta_1=c(\omega_\T)$
to obtain
\begin{align*}
|F_{2,J}(\bar{x})|
&\lesssim
M_{\ge \ell(J)}F_1(\bar{x})
+
\inf_{z \in Q(\bar{x},\ell(J))}
\sup_{\varepsilon>0}
|a(x,D-c(\omega_\T))[\chi_{\R^n \setminus Q(z,\varepsilon)}F_1](z)|.
\end{align*}
Let us set
\[
F_3(\bar{x})
:=
M F_1(\bar{x})
+
\sup_{\varepsilon>0}
|a(x,D-c(\omega_\T))[\chi_{\R^n \setminus Q(\bar{x},\varepsilon)}F_1](\bar{x})|
\]
for $\bar{x} \in \R^n$.
Here $M$ denotes the usual Hardy-Littlewood maximal operator.

In view of this result and the fact that $4\ell(J) \le \ell(I_s)$,
we obtain
\[
\int |F_{2,J}(y)|\,dy
\lesssim
\left|J \cap \bigcup_{s \in \T, \, |I_s|>2^n|J|}E_{s(2^n)}\right|
\cdot
\frac{1}{|J|}
\int_{J}F_3(y)\,dy.
\]
Hence it follows 
from the H\"{o}lder inequality 
that
\begin{align*}
\sum_J \int |F_{2,J}(y)|\,dy
&\lesssim
{\rm Dense}(\T)
\sum_{J \in {\cal J},|I_t|>2^n|J|} 
\sqrt{|J|\int_J F_3(y)^2\,dy}\\
&\lesssim
{\rm Dense}(\T)
\sqrt{
\sum_{J \in {\cal J},|I_t|>2^n|J|}
|J|
\int_{\R^n}F_3(y)^2\,dy
}.
\end{align*}

Since $M$ and $a(x,D)$ are both $L^2$-bounded,
we obtain
\[
\int_{\R^n}F_3(y)^2\,dy
\lesssim
\int_{\R^n}|F_1(y)|^2\,dy
\lesssim
|I_t|{\rm Size}(\T)^2
\]
from Lemma \ref{lem:080824-6}.
If we combine our observations,
we obtain 
\[
\sum_J \int |F_{2,J}(y)|\,dy
\lesssim
{\rm Dense}(\T){\rm Size}(\T)|I_t|,
\]
yielding the desired result.

\section{Self-extension}
\label{section:extension}
\noindent

Finally in the present paper,
we consider the self-extension of the main result.
With Theorem \ref{thm:main} established,
we can prove the following result
using the result in \cite{Grafakos-Tao-Terwilleger}.

\begin{theorem}
\label{thm:main4}
Suppose that $1<p<\infty$.
Then we have
\[
\left\|
\sup_{\xi \in \Q^n}|M_{-\xi}a(x,D)M_{\xi}f|
\right\|_p
\lesssim_p
\| f \|_p.
\]
\end{theorem}

\begin{proof}
The proof of Theorem \ref{thm:main4}
is completely the same as \cite{Grafakos-Tao-Terwilleger}
once we prove the basic estiate, Theorem \ref{thm:main2}.
We omit the details.
\end{proof}

\section*{Acknowledgement}
\noindent

The author is grateful to Professor E.~Terwilleger
for her encouragement.

\paragraph{Yoshihiro Sawano}
Department of Mathematics,\\ 
Gakushuin University, \\
1-5-1 Mejiro, Toshima-ku, \\
Tokyo 171-8588, Japan\\
e-mail\,:yosihiro@math.gakushuin.ac.jp

\end{document}